\documentclass[a4paper,leqno,12pt,draft]{amsart}
\usepackage{amsmath}
\usepackage{amssymb}
\usepackage{amsfonts}
\usepackage{amsthm}
\usepackage{enumerate}
\usepackage{comment}
\usepackage{cite}
\usepackage[dvips]{graphicx,color}
\usepackage{delarray}
\usepackage{ascmac}
\usepackage{fancybox}
\usepackage{lineno}
\usepackage[hidelinks,draft=false]{hyperref}
\theoremstyle{plain}
\newtheorem{theorem}{Theorem}[section]
\newtheorem{corollary}[theorem]{Corollary}
\newtheorem{definition}[theorem]{Definition}
\newtheorem{proposition}[theorem]{Proposition}

\newtheorem{remark}[theorem]{Remark}

\numberwithin{equation}{section}

\numberwithin{equation}{section}

\def\XXint#1#2#3{{\setbox0=\hbox{$#1{#2#3}{\int}$}
\vcenter{\hbox{$#2#3$}}\kern-.5\wd0}}

\begin{document}
\title[Regularity of extremal solutions]{Regularity of extremal solutions of semilinear elliptic equations with \texorpdfstring{$m$}{Lg}-convex nonlinearities}
\author{Kenta Kumagai}
\address{Department of Mathematics, Tokyo Institute of Technology}
\email{kumagai.k.ah@m.titech.ac.jp}
\date{\today}

\begin{abstract}
We consider the Gelfand problem in a bounded smooth domain $\Omega\subset \mathbb{R}^N$ with the Dirichlet boundary condition. We are interested in the boundedness of the extremal solution $u^*$. When the dimension $N\ge10$, it is known that a singular extremal solution can be constructed for the nonlinearity $f(u)=e^u$ and $\Omega=B_1$. When $3\le N\le 9$, Cabr\'e, Figalli, Ros-Oton, and Serra (2020) proved the following surprising result: the extremal solution $u^*$ is bounded if the nonlinearity $f$ is positive, nondecreasing, and convex.

In this paper, we succeed in generalizing their result to general $m$-convex nonlinearities. Moreover, we give a unified viewpoint on the results of previous studies by considering $m$-convexity. We provide a closedness result for stable solutions with $m$-convex nonlinearities. As a consequence, we provide a Liouville-type result and by using a blow-up argument, we prove the boundedness of extremal solutions.
\end{abstract}
\keywords{Semilinear elliptic equation, Extremal solution, Boundedness, Stable solution}
    \subjclass[2020]{35K57, 35B65, 35B35}

\maketitle

\raggedbottom

\section{Introdution}
Let $\Omega$ be a bounded domain of class $C^3$ and $f:\mathbb{R}\to\mathbb{R}$ be a positive, nondecreasing, locally Lipschitz, and superlinear function in the sense that
\begin{align}
\lim_{u\to \infty}\frac{f(u)}{u} = \infty. \notag
\end{align}
We consider the semilinear elliptic problem
\begin{equation}
\label{gelfandf}
\left\{
\begin{alignedat}{4}
 -\Delta u_\lambda&=\lambda f(u_{\lambda})&\hspace{2mm} &\text{in } \Omega,\\
u_\lambda&>0  & &\text{in } \Omega, \\
u_\lambda&=0  & &\text{on } \partial \Omega,
\end{alignedat}
\right.
\end{equation}
where $\lambda$ is a positive parameter.

Equation (\ref{gelfandf}) is known as the Gelfand problem. It was first studied by Barenblatt in relation to combustion theory in a volume edited by Gelfand \cite{Gelf}. Later, it was studied by many authors. 
We say that a solution $u_\lambda \in C^2_0(\Omega)$ of (\ref{gelfandf}) is stable if
\begin{align}
\int_{\Omega}\lambda f'(u_{\lambda})\xi^2\,dx\le\int_{\Omega}|\nabla\xi|^2\,dx\hspace{6mm}\text{for all $\xi \in C^\infty_0(\Omega)$}.\notag
\end{align}
If we define the associated energy $E$ as 
\begin{align}
E(u_\lambda):=\int_{\Omega}\left(\frac{|\nabla u_{\lambda}|^2}{2}-\lambda F(u_\lambda)\right) \,dx, \notag
\end{align}
where $F(u):=\int_{0}^{u}f(s)\,ds$, then the stability of a solution $u$ is interpreted as the nonnegativity of the second variation $E''$ at the critical point $u$. In paticular, any local minimizer of $E$ is a stable solution. 

The following theorem is the fundamental result to deal with this problem.

\begin{theorem}[see \cite{B,Br,Dup}]
Let $f$ be a positive, nondecreasing, locally Lipschitz, and superlinear function. Then there exists a constant $\lambda^* = \lambda^*(\Omega,N,f) \in(0,\infty)$ such that
\begin{itemize}
\item For $0<\lambda<\lambda^*$, there exists a minimal classical solution $u_\lambda\in C^2(\overline{\Omega})$ of $(\ref{gelfandf})$.
In particular, $u_\lambda$ is stable and $u_\lambda<u_{\lambda'}$ for $\lambda<\lambda'$.
\item For $\lambda>\lambda^*$, there exists no classical solution $u\in C^2(\overline{\Omega})$ of $(\ref{gelfandf})$.
\item For $\lambda=\lambda^*$, we define $u^*:=\lim_{\lambda \uparrow \lambda^* } u_\lambda$. Then the function $u^*$ is an $L^1$-weak solution in the sense that $u^*\in L^1(\Omega)$, $f(u^*)\mathrm{dist}(\cdot,\partial \Omega)\in L^1(\Omega)$, and 
\begin{equation}
-\int_{\Omega}u^*\Delta \xi \,dx = \int_{\Omega}\lambda^*f(u^*)\xi \,dx\hspace{8mm} \text{for all}\hspace{2mm}\xi\in C^2_0(\overline{\Omega}).\notag
\end{equation}
This solution is called the extremal solution of $(\ref{gelfandf})$.
\end{itemize} 
\end{theorem}  
We are interested in the regularity of the extremal solution. To deal with the regularity problem, the dimension $N$ is a very important factor. Indeed, when $N\ge10$, Joseph and Lundgren \cite{JL} constructed the singular extremal solution $u^*(x)=\log(1/|x|^2)$ for $f(u)=e^u$, $\lambda^*=2(N-2)$, and $\Omega=B_1$. On the other hand, when $N\le9$, Crandall and Rabinowitz \cite{CR} proved that the extremal solution
is bounded if the nonlinearity $f\in C^2(\mathbb{R})$ is a positive, increasing, and convex function which satisfies
\begin{equation}
\lim_{u\to \infty} \frac{f(u)f''(u)}{f'(u)^2}<\infty. \notag
\end{equation}
Typical examples are $f(u)=(1+u)^p$ or $f(u)=e^u$.
Based on these results, Brezis asked in \cite{B} whether extremal solutions are bounded for all $N\le9$, $\Omega$, and convex nonlinearities 
$f$.
   
In the last few decades, many studies have been done to prove the open problem (see \cite{Ned,cabrecapella,cabre2010,cabre2019,Vil,CR-O,CSS}). Finally, recently, Cabr\'e, Figalli, Ros-Oton, and Serra \cite{CFRS} solved the open problem positively. Moreover, they proved the boundedness of extremal solutions for general nonlinearities and convex domains. 

The case in which the nonlinearities and domains are not convex is more challenging problem. For this case, Cabr\'e \cite{cabre2010} showed the boundedness of extremal solutions when $N=2$. By extending his method, Castorina and Sanch\'on \cite{CS} showed it when $N\le4$ and $f(u)$ is convex after a large $u$. Sanch\'on \cite{san} also proved it when $N\le6$ and $\liminf_{u\to\infty}f(u)f''(u)/f'(u)^2>0$. Moreover, Aghajani \cite{A} proved it if the nonlinearities satisfy the following:
\begin{align}
\label{Aghajani}
\frac{1}{2}<\beta_{-}:=\liminf_{u\to\infty}\frac{f'(u)F(u)}{f(u)^2}\le \beta_{+}:=\limsup_{u\to\infty}\frac{f'(u)F(u)}{f(u)^2}<\frac{7}{10}\notag
\end{align}
or
\begin{equation}
\frac{1}{2}<\beta_{-}=\beta_{+}<\infty.\notag
\end{equation}

Our goal is to relax the convexity of nonlinearities. In order to achieve this goal, we focus on the framework of power convexity. For $m\in \mathbb{N}$, we say that a nonnegative function $f$ is $m$-convex if $f^m$ is convex. We note that if $f$ is $m$-convex, then $f$ is $l$-convex for all $m\le l$. Power convexity is a useful extension of convexity and it has become one of the well-known objects of study in elliptic and parabolic equations. For a systematic study of power convexity, see for instance \cite{i,Ishige}.

In this paper, we succeed in generalizing the result in \cite{CFRS} to all $m$-convex nonlinearities. More precisely, we prove the following: 
\begin{theorem}
\label{th1}
Let $N\le 9$, $m\in \mathbb{N}$, and $\Omega\subset \mathbb{R}^N$ be any bounded domain of class $C^3$. Assume that $f$ is a positive, nondecreasing, $m$-convex, and superlinear function. Then the extremal solution of $(\ref{gelfandf})$ is bounded. 
\end{theorem}
The following collorary follows immediately from Theorem \ref{th1}. 
\begin{corollary}
\label{cor}
Let $N\le 9$ and $\Omega\subset \mathbb{R}^N$ be any bounded domain of class $C^3$. Assume that $f \in C^2(\mathbb{R})$ is a positive, increasing, and superlinear function which satisfies 
\begin{equation}
\label{ratio}
\liminf_{u\to \infty} \frac{f(u)f''(u)}{f'(u)^2}>-\infty. 
\end{equation}
Then the extremal solution of $(\ref{gelfandf})$ is bounded.
\end{corollary}
We provide a unified viewpoint on the results of previous studies by considering $m$-convexity. As mentioned before, the ratio $f(u)f''(u)/f'(u)^2$ was introduced by Crandall and Rabinowitz \cite{CR} to state a condition that is much stronger than convexity. Then, after years of research, it was eliminated by Cabr\'e, Figalli, Ros-Oton, and Serra \cite{CFRS}. Corollary \ref{cor} revives an important feature of the ratio in the
boundedness of extremal solutions; and our results thus clarify the relation between the results in \cite{CR}, \cite{CS}, \cite{san}, and \cite{CFRS} through consideration of $m$-convexity.

The idea of the proof is as follows. Since a extremal solution $u^*\in L^1(\Omega)$ is approximated by smooth stable solutions, it is sufficient to provide an a priori $L^\infty$ estimate of classical stable solutions. In order to provide the estimate, we apply the method used in \cite{CFRS}. In \cite{CFRS}, the authors first provided an interior $L^\infty$ estimate and a global $W^{1,2+\gamma}$ estimate of the classical stable solutions with all nonlinearities. Thanks to the global $W^{1,2+\gamma}$ estimate, they constructed a closedness result for stable solutions with convex nonlinearities. As a consequence, they provided a Liouville-type result and by using a blow-up argument, they proved a boundary $L^\infty$ estimate. Our attempt is to extend the closedness result to a larger class of solutions. We define the class of $m$-convex functions and provide a new closedness result for stable solutions with $m$-convex nonlinearities. As a result, we prove the main theorem as an analogue of \cite{CFRS}.
\section{A closedness result for stable solutions with \texorpdfstring{$m$}{Lg}-convex nonlinearities}
The aim of this section is to provide a closedness result for solutions with $m$-convex nonlinearities. As mentioned in the introduction, this result is necessary to prove a Liouville-type result and apply a blow-up argument. 

Let $m$ be a natural number. We say that a nonnegative function $f:\mathbb{R}\to [0,\infty]$ is $m$-convex if $f^m$ is convex in 
$(-\infty, \sup_{f(t)<\infty} t)$ and define 
\begin{align}
C^m:=\left\{f:\mathbb{R}\to [0,\infty]:
\begin{array}{ll}
f\text{ is lower semi-continuous, nonnegative, }\\ 
\text{nondecreasing, and $m$-convex.} 
\end{array}
\right\}. \notag
\end{align}
For $f\in C^m$, we define $f'_{-}(t)$ as
\begin{align}
f'_{-}(t):=
\begin{cases}
\lim_{k\downarrow 0} \frac{f(t)-f(t-k)}{k} &\text{if}\hspace{2mm} f(t)<\infty, \\
\infty & \text{if}\hspace{2mm} f(t)=\infty.\\
\end{cases}
\notag
\end{align}
We note the following two properties. The first notation is that the inclusion $C^{m_1}\subset C^{m_2}$ holds for $m_1\le m_2$. Indeed, let $f\in C^{m_1}$ and define $F(g):=g^\frac{m_2}{m_1}$. Since $F$ and $f^{m_1}$ is convex, we have 
\begin{equation}
F\left(f^{m_1}(\frac{s_1+s_2}{2})\right)\le F\left(\frac{f^{m_1}(s_1)+f^{m_1}(s_2)}{2}\right)\le \frac{F(f^{m_1}(s_1))+F(f^{m_1}(s_2))}{2}\notag
\end{equation}
for all $s_1,s_2\in (-\infty, \sup_{f(t)<\infty} t)$. Therefore, we have $f\in C^{m_2}$. 
The second notation is that $C^m$ is closed under the scaling $f\mapsto af(b\cdot)$ for all $a>0$, $b>0$. This notation is important since we rescale the nonlinearity $f$ when we use the blow-up analysis.

Next, we give the definition of stable solutions. Let $\Omega$ be a domain and consider the equation
\begin{equation}
\label{eqf}
-\Delta u =f(u) \hspace{3mm} \text{in $\Omega$}.
\end{equation}
\begin{definition}
\label{stable}
Let $m\in \mathbb{N}$ and $f\in C^m$. Then we say that $u\in W^{1,2}_{\rm{loc}}(\Omega)$ is a weak solution of (\ref{eqf})
if $f(u)\in L^1_{\rm{loc}}(\Omega)$ and
\begin{equation}
\int_{\Omega}\nabla u \cdot \nabla \varphi \,dx = \int_{\Omega}f(u) \varphi \,dx \hspace{5mm} \text{for all $\varphi \in C^{\infty}_0 (\Omega)$}.\notag  
\end{equation}
Then, we say that $u$ is a stable solution if $f'_{-}(u)\in L^1_{\rm{loc}}(\Omega)$ and
\begin{equation}
\int_{\Omega}f'_{-}(u){\xi}^2\,dx \leq \int_{\Omega}{|\nabla \xi |}^2\,dx \hspace{5mm} \text{for all $\xi \in C^{\infty}_0 (\Omega)$}.\notag
\end{equation}
Moreover, we define
\begin{align}
S^m(\Omega):= \left\{u\in W^{1,2}_{\rm{loc}}(\Omega): u\text{ is a stable solution of (\ref{eqf}) for some }f\in C^m\right\}. \notag
\end{align}
\end{definition}
We note that if $u$ is a weak solution of (\ref{eqf}) for some $f\in C^m$, then $f(t)$ is finite for $t<\sup_{\Omega} u$. It follows from $f(u)\in L^1_{\rm{loc}}(\Omega)$. In particular, if $\{u(x)=\sup_{\Omega} u\}$ has a positive measure, we have $f(\sup_{\Omega}u)<\infty$.

The following proposition is the closedness result.

\begin{proposition}
\label{closedness}
Let $m\in \mathbb{N}$, $\gamma>0$, and let $\{u_k\}_{k\in \mathbb{N}}\subset S^m(\Omega)$ be a $W^{1,2+\gamma}$ locally bounded sequence. Then, there exist a subsequence $\{u_{k_j}\}_{j\in \mathbb{N}}\subset \{u_k\}_{k\in \mathbb{N}}$ and $u\in S^m(\Omega)$ such that $u_{k_j}\to u$ in $W^{1,2}_{\rm{loc}}(\Omega)$ as $j\to \infty$. 
\end{proposition}
\begin{remark}
\label{remark}
\rm{We note that if $u$ is a weak solution of (\ref{eqf}) for some $f\in C^m$, then $f(t)$ is finite for $t<\sup_{\Omega} u$. It follows from $f(u)\in L^1_{\rm{loc}}(\Omega)$. In particular, if $\{u(x)=\sup_{\Omega} u\}$ has a positive measure, we have $f(\sup_{\Omega}u)<\infty$.}
\end{remark}
\begin{proof}
By assumption, we have a $W^{1,2+\gamma}$ locally bounded sequence $u_k\in S^m(\Omega)$ of stable solutions of $-\Delta u_k=f_k(u_k)$ with $f_k\in C^m$.
Thanks to the $W^{1,2+\gamma}$ locally boundedness of $u_k$, by using the Rellich-Kondrachov theorem and a diagonal argument, we verify that there exists a function $u\in L^{2}_{\rm{loc}}(\Omega)$ such that up to subsequence if necessary, $u_k\to u$ in $L^2_{\rm{loc}}(\Omega)$ in the sense that for any domain $\Omega'\subset\subset \Omega$, $u_k\to u$ in $L^2(\Omega')$. Furthermore, thanks to an interpolation inequality and the $W^{1,2+\gamma}$ locally boundedness of $u_k$, we get $u\in W^{1,2}_{\rm{loc}}(\Omega)$, $u_k(x)\to u(x)$ in $W^{1,2}_{\rm{loc}}(\Omega)$, and $u_k\to u$ for a.e. $x\in \Omega$ (if necessary, by taking a subsequence).

The following proof generalizes the method of \cite[Theorem 4.1]{CFRS} by focusing on the convexity of $f_{k}^m$ instead of the convexity of $f_{k}$. There are three main differences as follows. The first is that in Step 1, we prove the boundedness of $f_k(l)$ for all $l<\sup_{\Omega}u$ without using the condition of $f$. The second is that in Step 2, we use the $m$-convexity of $f_k$ and apply an algebraic argument to show that $\lVert f_k (u_k)-f_k (u_k-\delta)\rVert_{L^1(\Omega')}\to 0$ uniformly on $k$ as $\delta\to 0$ for all subdomains $\Omega'\subset\subset\Omega$. The last point is that in Step 3, we prove the inequality (\ref{stab}) by using the $(m+1)$-convexity of $f_k$ and show that $f^m_k(u_k)$ on the left side of (\ref{stab}) can be replaced by $f^m_k(u_k-2\delta)$. 

\vspace{10pt}
\noindent
\textbf{Step 1.} A compactness estimate on $f_k$.
  
Set $L:=\sup_{\Omega}u \in (-\infty,\infty]$ and let $l<L$. We claim that
\begin{align}
\label{limsup}
\limsup_{k\to \infty} f_k(l)<\infty.
\end{align}
Indeed, let $l<L$ and $\Omega'\subset\subset \Omega$ be a domain such that $\{u(x)>l\}\cap \Omega'$ has a positive measure.
We choose a nonnegative cut-off function $\eta\in C^\infty_{0}(\Omega)$ such that $\eta=1$ in $\Omega'$.
Since $u_k(x)\to u(x)$ for a.e. $x\in \Omega$, by applying Fatou's Lemma to the sequence $1_{\{u_k>l\}\cap \Omega'}$, we get
\begin{equation}
\label{menseki}
\liminf_{k\to\infty}|\{u_k>l\}\cap \Omega'|\ge|\{u(x)>l\}\cap \Omega'|>0.
\end{equation}
On the other hand, Since $f$ is nonnegative and nondecreasing, we have
\begin{align}
\label{bound of fk}
\begin{split}
f_k(l)&\le\frac{1}{|\{u_k>l\}\cap \Omega'|}\int_{\{u_k>l\}\cap \Omega'}f_k(u_k) \, dx\\
&\le \frac{1}{|\{u_k>l\}\cap \Omega'|}\int_{\Omega}f_k(u_k)\eta\, dx \\
&\le\frac{1}{|\{u_k>l\}\cap \Omega'|}\int_{\Omega}\nabla u_k \cdot \nabla \eta \, dx \le\frac{C}{|\{u_k>l\}\cap \Omega'|} \lVert\nabla u_k\rVert_{L^2(\Omega)}
\end{split}
\end{align}
for some constant $C$ independent of $k$ and all $k$ suffficiently large. Combining (\ref{menseki}) and (\ref{bound of fk}), we obtain (\ref{limsup}).

Furthermore, by the convexity of $f^{m}_{k}$, we have for all $k$ suffficiently large,
\begin{equation}
(f^{m}_{k})'_{-}(l)<\frac{f^{m}_{k} (l+\delta)-f^{m}_{k} (l)}{\delta}<\limsup_{k\to \infty}\frac{f^{m}_{k}(l+\delta)}{\delta}<\infty \hspace{5mm} \text{for all $l<L$}.\notag
\end{equation}
Since $m$-convex functions are $(m+1)$-convex, in a similar way, it follows that $\limsup_{k\to\infty}(f^{m+1}_{k})'_{-}(l)<\infty$ for all $l<L$.
Therefore, $f^{m}_k$ is uniformly bounded and equicontinuous. By Ascoli Arzela's theorem and a diagonal argument, there exist the function $g: (-\infty,L)\to \mathbb{R}$ such that $f^{m}_k\to g$ locally uniformly on $(-\infty,L)$.
Define $g(L):=\lim_{l\uparrow L}g(l)$, $g(l):=\infty$ for $l>L$, and $f:=g^{\frac{1}{m}}$. Then it is easy to check $f\in C^m$ and $f_k\to f$ locally uniformly on $(-\infty,L)$.

\vspace{10pt}
\noindent
\textbf{Step 2.} $-\Delta u = f(u)$ in $\Omega$.

For every nonnegative test function $\xi\in C^{0,1}_{0}(\Omega)$, we have 
\begin{align}
\int_{\Omega} \nabla u \cdot \nabla \xi \,dx = \lim_{k \to \infty} \int_{\Omega} \nabla u_k \cdot \nabla \xi \,dx =\lim_{k \to \infty}\int_{\Omega} f_k(u_k) \xi \,dx. \notag
\end{align}
We claim that $f_k(u_k) \to f(u)$ for a.e. $x \in\{u<L\}$ as $k\to \infty$.
Indeed, let $x \in \{u<L\}$ and let $l$ be a positive constant satisfying $u(x)<l<L$. Then we have
\begin{align}
|f^{m}_k(u_k(x))-g(u(x))| \le& |g(u(x))-f^{m}_k(u(x))|+|f^{m}_k(u(x))-f^{m}_k(u_k(x))| \notag \\
\le & |g(u(x))-f^{m}_k(u(x))|+(f^{m}_k)'_{-}(l)|u(x)-u_k(x)| \notag \\
\to &0 \hspace{2mm}\text{as $k \to \infty$}\notag. 
\end{align}
Thus, this claim holds.

In the following, $\eta \in C^\infty_0(\Omega)$ denotes a nonnegative cut-off function such that $\eta=1$ on the support of $\xi$.

\vspace{10pt}
\noindent
\textbf{Case 1.} $L=\infty$. 

We have
\begin{align}
\int_{\rm{supp} (\xi)} f_k (u_k) u_k\, dx \le \int_{\Omega} f_k (u_k) u_k \eta\, dx 
= \int_{\Omega} \nabla u_k \cdot \nabla (u_k \eta)\,dx \le C \notag
\end{align}
for some constant $C$ independent of $k$, where the last bound follows from the $W^{1,2}$ locally boundedness of $u_k$. We take a continuous function $\varphi : \mathbb{R} \to [0,1]$ such that $\varphi=0$ on $(-\infty,0]$ and $\varphi=1$ on $[1,\infty)$. Then we deduce that
\begin{align}
\label{ineqj}
\begin{split}
\int_{\rm{supp}(\xi)} \varphi(u_k-j) f_k (u_k) \,dx \le&\int_{\rm{supp}(\xi)\cap\{u_k>j\}} f_k (u_k) \,dx \\
\le& \frac{1}{j}\int_{\rm{supp} (\xi) \cap\{u_k>j \}} f_k (u_k)u_k \,dx \le \frac{C}{j}.
\end{split}
\end{align}
Therefore, by Fatou's Lemma, we also have
\begin{align}
\label{fatou}
\int_{\rm{supp}(\xi)} f(u)\varphi(u-j)\, dx \le \frac{C}{j}\hspace{5mm}\text{for all $j\ge 1$}. 
\end{align}
Especially, we get $f(u) \in L^1_{\rm{loc}}(\Omega)$.
Furthermore, since
\begin{align}
&f_k (u_k)[1-\varphi(u_k-j)]\le f_k (j+1) \le C_j,  \notag\\
&f_k (u_k)[1-\varphi(u_k-j)] \to f(u)[1-\varphi(u-j)] \hspace{5mm}\text{for a.e. $x\in \Omega$ as $k\to \infty$}, \notag
\end{align}
by combing (\ref{ineqj}) and (\ref{fatou}) and applying the dominated convergence theorem, we have
\begin{align}
\int_{\rm{supp}(\xi)}&|f_k (u_k)- f(u)|\,dx \notag \\
\le& \int_{\rm{supp}(\xi)}|f_k (u_k)(1-\varphi(u_k-j))- f(u)(1-\varphi(u-j))|\, dx \notag \\
 &+ \int_{\rm{supp}(\xi)}|f_k (u_k)\varphi(u_k-j)+f(u)\varphi(u-j)|\,dx\notag \\
\le&  \frac{C}{j} + o(1). \notag
\end{align}                 
By letting first $k\to \infty$ and then $j\to \infty$, we get the result. 

\vspace{10pt}
\noindent
\textbf{Case 2.} $L<\infty$.

Let $\delta>0$. Since $(u_k-L-\delta)_{+} \ge \delta $ in $\{u_k>L+2\delta\}$, we have
\begin{align}
\delta&\int_{\rm{supp}(\xi) \cap \{u_k>L+2\delta\}}f_k (u_k)\,dx  \notag \\
&\le\int_{\rm{supp}(\xi) \cap \{u_k>L+2\delta\}}f_k (u_k)(u_k-L-\delta)_{+}\,dx \notag \\
&\le \int_{\Omega}f_k(u_k)(u_k-L-\delta)_{+}\eta\,dx \notag\\
&\le \int_{\Omega}\left(|\nabla u_k|^2\eta+\nabla u_k \cdot \nabla\eta(u_k-L-\delta)_{+}\right)1_{\{u_k>L+\delta\}}\,dx. \notag
\end{align}
Since $u_k$ are uniformly bounded in $W^{1,2+\gamma}(\rm{supp}(\eta))$ and $1_{\{u_k>L+\delta\}} \to 0$ and
$(u_k-L-\delta)_{+} \to 0$ for a.e. $x\in \Omega$ as $k\to \infty$, we deduce from H\"older's inequality that,
\begin{align}
\lim_{k\to \infty}\int_{\rm{supp}(\xi)\cap\{u_k>L+2\delta\}} f_k(u_k)\,dx=0. \notag
\end{align}
Hence, we get
\begin{align}
\label{estimate}
\begin{split}
\int_{\Omega}f_k(u_k)\xi \,dx=&\int_{\Omega \cap \{u_k\le L+2\delta\}}f_k(u_k)\xi \,dx + o(1)  \\
\le&\hspace{2mm} \mathrm{I} + \mathrm{II}+o(1),
\end{split} 
\end{align}
where
\begin{align}
&\mathrm{I} = \int_{\Omega\cap\{u_k\le L+2\delta\}}f_k(u_k-3\delta)\xi \,dx , \notag \\
&\mathrm{II} = \int_{\Omega\cap\{u_k\le L+2\delta\}}\left(f_k(u_k)-f_k(u_k-3\delta)\right)\xi \,dx. \notag
\end{align}
We note that $f_k(u_k-3\delta) \to f(u-3\delta)$ for a.e. $x\in \Omega$ and $f_k (u_k-3\delta)\le f_k (L-\delta)\le C_\delta$
in $\rm{supp}(\xi) \cap \{u_k\le L+2\delta\}$ for some constant $C_{\delta}$ depending on $\delta$ but not on $k$. Thus we deduce from the dominated convergence theorem that
\begin{align}
\label{estimatei} 
\mathrm{I} = \int_{\Omega}f(u-3\delta)\xi \,dx + o(1).
\end{align}
On the other hand, we prove $\mathrm{|II|}<C\delta$ for some $C$ independent of $k$ and $\delta$. Indeed, let $k>0$. We assume $\{x: u_k(x)=\sup_{\Omega}u_k \}$ has a positive measure.
As mentioned in Remark \ref{remark}, we have $f_k(\sup_{\Omega} u_k)<\infty$. Thus we have
\begin{align}
f_k(u_k(x))-f_k(u_k(x) - 3\delta) =& \frac{f_k(u_k(x))-f_k(u_k(x) - 3\delta)}{f^{m}_k(u_k(x))-f^{m}_k(u_k(x) - 3\delta)}\int^{u_k(x)}_{u_k(x)-3\delta} (f^{m}_{k})'_{-} (s) \,ds \notag\\ 
\le& 3m\delta\frac{f_k(u_k(x))-f_k(u_k(x) - 3\delta)}{f^{m}_k (u_k(x))-f^{m}_k(u_k(x) - 3\delta)}f^{m-1}_k (u_k) (f_k)'_{-}(u_k) \notag\\
\le& 3m\delta (f_k)'_{-}(u_k(x)) \notag
\end{align}
for all $x\in \Omega$. In a similar way, if $\{x: u_k(x)=\sup_{\Omega}u_k\}$ is a null set, the above estimate holds for a.e. $x \in \Omega$. 
Thanks to the stability of $u_k$, we get
\begin{align}
\label{estimateii}
|\mathrm{II}| = C\delta \int_{\Omega\cap\{u_k\le L+2\delta\}}(f_{k})'_{-}(u_k)\xi\,dx <C\delta
\end{align}
for some constant $C$ independent of $k$ and $\delta$. Combing (\ref{estimate}), (\ref{estimatei}), and (\ref{estimateii}), we have
\begin{align}
\int_{\Omega}f(u-3\delta)\xi\,dx -C\delta&\le\liminf_{k\to \infty} \int_{\Omega}f_k(u_k)\xi\,dx \notag \\
&\le\limsup_{k\to \infty} \int_{\Omega}(f_k)(u_k)\xi\, dx\le\int_{\Omega}f(u-3\delta)\xi\,dx+C\delta,\notag
\end{align}
for some constant $C$ independent of $\delta$. Since $\delta$ is arbitrary, by applying the monotone convergence theorem,\footnote{It is sufficient to prove only when $\xi>0$. Indeed, every $\xi\in C^\infty_{0}(\Omega)$ is represented as a difference between two positive functions belonging to $C^{0,1}_{0}(\Omega)$.} we have $f(u)\in L^1_{\rm{loc}}(\Omega)$ and
\begin{align}
\lim_{k\to \infty} \int_{\Omega}f_k(u_k)\xi \,dx =\int_{\Omega}f(u) \xi \,dx.\notag
\end{align}
\vspace{10pt}
\noindent
\textbf{Step 3.} The stability of $u$.

Define $T:=\sup_{f(t)=0}t$. Then without loss of generality, we assume $T<L$. Let $0<\varepsilon<\min\{1, (L-T)/8\}$. Since $f_k\to f$ locally uniformly on $(-\infty, L)$, there exists a large number $K>0$ such that
\begin{align}
\label{a}
f_k(T+2\varepsilon)\ge\frac{f(T+2\varepsilon)}{2}>0\hspace{5mm} \text{for all $k>K$}. 
\end{align}
Let $j\in \mathbb{N}$ and $0<\delta<\varepsilon$. We denote $L_j:=\min\{j,L\}$ and $E_k:=\{T+4\varepsilon<u_k<L_{j}+\delta \}$. We note that since $f_k$ is $m$-convex (and thus $(m+1)$-convex), we have
\begin{equation}
f_k^{m+1} (u_k-2\varepsilon)-f_k^{m+1}(u_k-2\varepsilon-\delta)\le (m+1)\delta f_k^{m}(u_k)(f_k)'_{-}(u_k).\notag
\end{equation}
Thanks to this notation and stability inequality, we have for all $\xi\in C^{\infty}_{0}(\Omega)$ and $k>K$,
\begin{align}
\label{stab}
\begin{split}
\int_{E_k} \frac{f^{m+1}_k(u_k-2\varepsilon)-f^{m+1}_k(u_k-2\varepsilon-\delta)}{(m+1)\delta f_k^{m}(u_k)}\xi^2\,dx \le \int_{\Omega} |\nabla \xi|^2 \,dx.
\end{split}
\end{align}
Our attempt is to replace $f^m_k(u_k)$ on the left side of (\ref{stab}) with $f^m_k(u_k-2\delta)$. In order to justify it, we claim that we have
\begin{align}
\label{aa}
\begin{split}
\int_{E_k}\frac{f^{m+1}_k(u_k-2\varepsilon)-f^{m+1}_k(u_k-2\varepsilon-\delta)}{(m+1)\delta}\cdot&\frac{f_k^{m}(u_k)-f_k^{m}(u_k-2\delta)}{f_k^{m}(u_k)f_k^{m}(u_k-2\delta)}\xi^2\,dx\\
&<C\delta
\end{split}
\end{align}
for all $k>K$ and some constant $C$ independent of $\delta$ and $k$. Indeed, since $f_k$ is $(m+1)$-convex and $(f^{m+1}_k)'_{-}(l)$ is uniformly bounded for all $l<L$, thanks to (\ref{a}), we get for all $k>K$,
\begin{equation}
\label{b}
\frac{f^{m+1}_k(u_k-2\varepsilon)-f^{m+1}_k(u_k-2\varepsilon-\delta)}{(m+1)\delta f_k(u_k)f_k^{m}(u_k-2\delta)}<C(f^{m+1}_k)'_{-}(u_k-2\varepsilon)<C, 
\end{equation}
where $C$ is a constant independent of $\delta$ and $k$.
On the other hand, since $f_k$ is $m$-convex, we have for all $k>K$,
\begin{equation}
\label{c}
\int_{E_k}\frac{f_k^{m}(u_k)-f_k^{m}(u_k-2\delta)}{f_k^{m-1}(u_k)}\xi^2\,dx\le C\delta\int_{\Omega}(f_k)'_{-}(u_k)\xi^2\, dx<C\delta
\end{equation}
for some constant $C$ independent of $\delta$ and $k$. Therefore this claim holds by combing (\ref{b}) and (\ref{c}).
Moreover, thanks to $f_k\to f$ locally uniformly on $(-\infty, L)$ and $u_k\to u$ a.e. in $\Omega$, By combing (\ref{aa}) and (\ref{stab}) and applying Fatou's lemma, we have
\begin{align}
\int_{\{T+6\varepsilon<u\le L_j \}}&\frac{f^{m+1}(u-2\varepsilon)-f^{m+1}(u-2\varepsilon-\delta)}{(m+1)\delta f^{m}(u)}\xi^2\,dx\notag\\
&\le\int_{\{T+6\varepsilon< u\le L_j \}}\frac{f^{m+1}(u-2\varepsilon)-f^{m+1}(u-2\varepsilon-\delta)}{(m+1)\delta f^{m}(u-2\delta)}\xi^2\,dx\notag\\
&\le\liminf_{k\to\infty}\int_{E_k} \frac{f^{m+1}_k(u_k-2\varepsilon)-f^{m+1}_k(u_k-2\varepsilon-\delta)}{(m+1)\delta f_k^{m}(u_k-2\delta)}\xi^2\,dx\notag\\
&\le C\delta+\int_{\Omega} |\nabla \xi|^2 \,dx\notag
\end{align}
for some constant $C$ independent of $\delta$ and $k$. Since 
\begin{equation}
\frac{f^{m+1}(t-2\varepsilon)-f^{m+1}(t-2\varepsilon-\delta)}{(m+1)\delta}\uparrow f^m(t-2\varepsilon)f'_{-}(t-2\varepsilon) \hspace{5mm} \text{as $\delta\downarrow0$ for $t<L_j$}\notag
\end{equation}
and
\begin{equation}
f^m(t-2\varepsilon)f'_{-}(t-2\varepsilon) \uparrow f^m(t)f'_{-}(t) \hspace{5mm} \text{as $\varepsilon\downarrow0$ for $t<L_j$}\notag,
\end{equation}
by letting first $\delta\to 0$ and then $\varepsilon\to0$ and $j\to\infty$, we get 
\begin{equation}
\int_{\{T<u\}} f'_{-}(u)\xi^2\, dx\le \int_{\Omega}|\nabla \xi|^2\,dx. \notag
\end{equation}
Since $f'_{-}(u)=0$ in $\{u\le T\}$, we get the result.
\end{proof}
\section{A Liouville-type result}
In this section, we provide a Liouville-type result. To begin with, we denote $B_r:=\{|x|<r\}$, $\mathbb{R}^{N}_{+}:=\mathbb{R}^{N}\cap \{x_N>0\}$, and $B_{r}^{+}:= B_r \cap \{x_N>0\}$.
The following proposition states a Liouville-type result. Thanks to Proposition \ref{closedness}, this proposition is proved as a modification of \cite[Proposition 6.3]{CFRS}. 
\begin{proposition}
\label{LT}
Assume $3\le N \le 9$ and let $m\in \mathbb{N}$. There exists a dimensional constant $\alpha_{N}>0$ such that the following holds.
Assume that a nonnegative function $u\in W^{1,2}_{\rm{loc}}(\overline{\mathbb{R}^{N}_{+}}) \cap C^{0}_{\rm{loc}}(\mathbb{R}^{N}_{+})$ satisfies $u_R \in S^{m}(B^{+}_{2})$ and $u_R = 0$ on $\partial B^{+}_{2}\cap\{x_N =0 \}$ in the trace sense for all $R\ge 1$, where $u_R (x):=u(Rx)$.
Suppose in addition that, for some $\alpha \in (0,\alpha_N)$ and $\gamma>0$, we have 
\begin{align}
\label{ass}
\lVert \nabla u_R \rVert_{L^{2+\gamma}(B^{+}_{3/2})} \le C_1 \lVert \nabla u_R \rVert_{L^{2}({B^{+}_2})}\le C_2 R^\alpha \hspace{2mm}\text{for all $R\ge1$}
\end{align}
with constants $C_1$ and $C_2$ independent of $R$, and that $u$ satisfies
\begin{equation}
\label{cp}
\int_{\mathbb{R}^{N}_{+}} (\{ (N-2)\eta +2x \cdot \nabla \eta \}\eta |\nabla u|^2 - 2(x\cdot \nabla{u})\nabla u\cdot \nabla(\eta^2)-|x\cdot \nabla u|^2|\nabla \eta|^2) \,dx \le 0 
\end{equation}
for all $\eta \in C^{0,1}_0 (\overline{\mathbb{R}^{N}_{+}})$. Then, $u\equiv 0$.
\end{proposition}
\begin{proof}
For the reader's convenience, we note the sketch of the proof as follows. We first use the appropriate function as a test function in (\ref{cp}) and have an inequality, which tells us that the radial derivative of $u$ in a half-ball is controlled by the gradient term in a half-annulus.
Next, we prove the following statement: there exists a dimensional constant $C>0$ such that if for some $R\ge1$ we have  
\begin{equation}
\int_{B_1^+}|\nabla u_R|^2\,dx \ge \frac{1}{2}\int_{B_2^+}|\nabla u_R|^2\,dx \notag,
\end{equation}
then
\begin{equation}
\label{RI}
\int_{B_{3/2}^+}|\nabla u_R|^2\,dx\le C\int_{B_{3/2}^+ \setminus B_{1}^+}|x|^{-N}|x\cdot\nabla u_R|^2\,dx. 
\end{equation}
This statement tells us that the gradient term in a half-ball is controlled by the radial derivative of $u$ in a half-annulus under the condition that the mass of $|\nabla u|^2$ in a half-annulus is not large with respect to that in a half-ball.
By combing the previous two results, using an iteration argument and controlling the gradient term by using (\ref{ass}), we have $x\cdot u_R\equiv 0$ in $B^{+}_{2}$ for all $R\ge1$. As a result, we have $u_R\equiv 0$ in $B^{+}_{2}$ for all $R\ge1$\footnote{This fact can be proved by imitating the following method.} and we get the result. Therefore, it is sufficient to prove (\ref{RI}).
      
In this proof, we assume by contradiction the existence of a sequence $v_k:=u_{R_k}/\lVert\nabla u_{R_k} \rVert_{L^2(B^{+}_{3/2})} \in W^{1,2}_{\rm{loc}}(\overline{\mathbb{R}^{N}_{+}})\cap C^{0}_{\rm{loc}}(\mathbb{R}^{N}_{+})$ with $v_k=0$ on $\partial B^+_2 \cap\{x_N=0\}$ such that
\begin{equation}
\int_{B_1^{+}}|\nabla v_k|^2\, dx  \ge \frac{1}{2}\int_{B_2^{+}}|\nabla v_k|^2\, dx, \notag
\end{equation}
\begin{equation}
\int_{B_{3/2}^{+}}|\nabla v_k|^2\, dx =1,\hspace{5mm}\text{and}\hspace{5mm} \int_{B_{3/2}^{+}\setminus B_1^{+}}|x|^{-N}|x\cdot\nabla v_k|^2 \,dx \to 0 \hspace{2mm} \text{as $k\to\infty$. \notag}
\end{equation}
Since 
\begin{equation}
\lVert \nabla v_k \rVert_{L^{2+\gamma}(B^{+}_{3/2})}\le C_1\lVert \nabla v_k \rVert_{L^{2}({B^{+}_2})}\le 2C_1 \notag
\end{equation}
and
\begin{equation}
\lVert  v_k \rVert_{L^{2+\gamma}({B^{+}_{3/2}})}\le C\lVert \nabla v_k \rVert_{L^{2+\gamma}({B^{+}_{3/2}})}\le C\hspace{2mm} \text{by the poincar\'e inequality}, \notag
\end{equation}
similary to the beginning of the proof of Proposition \ref{closedness}, We have $v_k\to v$ in $W^{1,2}(B_{3/2}^{+})$ for some $v\in W^{1,2}(B_{3/2}^{+})$ by taking a subsequence if necessary.
Then we have
\begin{equation}
\int_{B_{3/2}^{+}}|\nabla v|^2\,dx=1\hspace{5mm}\text{and} \hspace{5mm} x\cdot\nabla v\equiv0\hspace{2mm} \text{in}\hspace{2mm}B_{3/2}^{+}\setminus B_{1}^{+}.\notag
\end{equation}
Moreover, we verify $v\in S^m(B_{3/2}^{+})$ by using Proposition \ref{closedness} and $v=0$ on $\{x_N=0\}\cap B_{3/2}^+$ by the continuity of the trace operator.
In particular, $v$ is a weak solution of $-\Delta v = g(v)$ in $B^{+}_{3/2}$ with some $g\in C^m$.
Thanks to the $0$-homogenity of $v$, we know that $\Delta v$ is $(-2)$-homogeneous and $g(v)$ is $0$-homogeneous. Thus we have $g(v)\equiv 0$. In particular, $v$ is a harmonic $0$-homogeneous function in $B_{3/2}^{+}\setminus{B_{1}^+}$ satisfying $v=0$ on $\{x_N=0\}\cap B_{3/2}^+$. As a consequence, $v$ takes its infimum at an interior point. Thanks to the strong maximum principle, we get $v\equiv0$ in $B_{3/2}^{+}\setminus{B_{1}^+}$. Thanks to the superharmonicity of $v$ and the strong maximum principle, we have $v\equiv0$ in $B_{3/2}^{+}$. This result contradicts $\int_{B^{+}_{3/2}}|\nabla v|^2\,dx=1$.  
\end{proof}   
\section{A blow-up argument}
In this section, we provide a priori $L^\infty$ estimate of $u$ by using a blow-up argument. At first, we introduce the notion of a small deformation of a half-ball.
\begin{definition}[see \cite{CFRS}] Let $\vartheta\ge 0$. We define that a domain $\Omega\subset\mathbb{R}^N$ is a $\vartheta$-deformation of $B^{+}_2$ if $\Omega = \Phi(B^{+}_2)$
for some $\Phi \in C^3(B_2;\mathbb{R}^{N})$ such that $\Phi(0)=0$, $D\Phi(0)=\rm{Id}$, and
\begin{equation}
\lVert D^2 \Phi \rVert_{L^{\infty}(B_2)} +\lVert D^3 \Phi \rVert_{L^{\infty}(B_2)}\le\vartheta, \notag
\end{equation}
where the norms of $D^2 \Phi$ and $D^3 \Phi$ are computed with respect to the operator norm.
\end{definition}
We note that given a bounded $C^3$ domain, we can cover its boundary with finite small balls. Then by rescaling and rotating these balls, we can regard its boundary as a finite union of $\vartheta$-deformations of $B^{+}_2$ with $\vartheta$ sufficiently small.
Therefore, it is enough to provide an a priori $L^\infty$ bound only if $\Omega$ is a small deformation of $B^{+}_2$.

The following theorem states an a priori $L^\infty$ estimate of $u$ and we prove it by modifying the proof of \cite[Theorem 6.1]{CFRS}.

\begin{theorem} [see \cite{CFRS}]
\label{blowup}
Let $3\le N \le 9$, $0\le\vartheta\le\frac{1}{100}$, $m\in \mathbb{N}$, and $\Omega\subset\mathbb{R}^N$ be a $\vartheta$-deformation of $B^+_2$. Assume that $u\in C^2(\overline{\Omega \cap B_1})$ is a nonnegative stable solution of
\begin{equation}
\label{eqx}
-\Delta u=f(u)\hspace{5mm}\text{in $\Omega\cap B_1$}\hspace{5mm}\text{and}\hspace{5mm}u=0\hspace{5mm}\text{on $\partial \Omega \cap B_1$} \notag
\end{equation} 
for a positive, nondecreasing, $m$-convex, and superlinear function $f$. Then, there are some constants $\alpha=\alpha(m,N)>0$ and $C=C(m,N)>0$ such that
\begin{equation}
\lVert u \rVert_{C^{\alpha}(\overline{\Omega}\cap B_{1/2})}\le C \lVert u \rVert_{L^1(\Omega\cap B_1)}.\notag
\end{equation}
\end{theorem}
\begin{proof}
In the proof of \cite[Theorem 6.1]{CFRS}, the convexity of $f$ is used only to apply the closedness result for $S^1(\Omega)$ and a Liouville-type result.
By replacing $S^1(\Omega)$ with $S^m(\Omega)$ and using Proposition \ref{closedness} and Proposition \ref{LT}, We can prove this theorem.
\end{proof}
\begin{proof}[Proof of Theorem \ref{th1} and Corollary \ref{cor}] 
Let $u_\lambda$ be the minimal solution of (\ref{gelfandf}) for some $\lambda<\lambda^*$ and $f\in C^m$. Since $C^m$ is closed under scaling, by applying Theorem \ref{blowup}, covering argument, and an interior $L^{\infty}$ estimate of $u_\lambda$ (see \cite[Theorem 1.2]{CFRS}), we get
\begin{equation} 
\lVert u_{\lambda} \rVert_{L^{\infty}(\Omega)}\le C \lVert u_{\lambda} \rVert_{L^1(\Omega)} \notag
\end{equation}
for some constant $C>0$ dependent only on $N$, $m$, and $\Omega$.
By applying the monotone convergence theorem, we prove Theorem \ref{th1}.

If $f\in C^2(\mathbb{R})$ satisfies (\ref{ratio}), there exists a large number $m\in \mathbb{N}$ such that $f''f+(m-1)f'^2>0$. Since
\begin{equation}
(f^m)''=mf^{m-2}((m-1)f'^2+f''f)>0,\notag
\end{equation}
the function $f$ is $m$-convex and this result follows from Theorem \ref{th1}.
\end{proof}
\section*{Acknowledgments}
The author would like to thank my supervisor, Associate Professor Michiaki Onodera, for his valuable advice.
\bibliographystyle{plain}
\bibliography{regularity_of_extremal_solutions}

\begin{thebibliography}{10}

\bibitem{A}
Asadollah Aghajani.
\newblock Regularity of extremal solutions of semilinear elliptic problems with
  non-convex nonlinearities on general domains.
\newblock {\em Discrete Contin. Dyn. Syst.}, 37(7):3521--3530, 2017.

\bibitem{B}
Haim Brezis.
\newblock Is there failure of the inverse function theorem?
\newblock In {\em Morse theory, minimax theory and their applications to
  nonlinear differential equations}, volume~1 of {\em New Stud. Adv. Math.},
  pages 23--33. Int. Press, Somerville, MA, 2003.

\bibitem{Br}
Haim Brezis and Juan~Luis V\'{a}zquez.
\newblock Blow-up solutions of some nonlinear elliptic problems.
\newblock {\em Rev. Mat. Univ. Complut. Madrid}, 10(2):443--469, 1997.

\bibitem{cabre2010}
Xavier Cabr\'{e}.
\newblock Regularity of minimizers of semilinear elliptic problems up to
  dimension 4.
\newblock {\em Comm. Pure Appl. Math.}, 63(10):1362--1380, 2010.

\bibitem{cabre2019}
Xavier Cabr\'{e}.
\newblock A new proof of the boundedness results for stable solutions to
  semilinear elliptic equations.
\newblock {\em Discrete Contin. Dyn. Syst.}, 39(12):7249--7264, 2019.

\bibitem{cabrecapella}
Xavier Cabr\'{e} and Antonio Capella.
\newblock Regularity of radial minimizers and extremal solutions of semilinear
  elliptic equations.
\newblock {\em J. Funct. Anal.}, 238(2):709--733, 2006.

\bibitem{CFRS}
Xavier Cabr\'{e}, Alessio Figalli, Xavier Ros-Oton, and Joaquim Serra.
\newblock Stable solutions to semilinear elliptic equations are smooth up to
  dimension 9.
\newblock {\em Acta Math.}, 224(2):187--252, 2020.

\bibitem{CR-O}
Xavier Cabr\'{e} and Xavier Ros-Oton.
\newblock Regularity of stable solutions up to dimension 7 in domains of double
  revolution.
\newblock {\em Comm. Partial Differential Equations}, 38(1):135--154, 2013.

\bibitem{CSS}
Xavier Cabr\'{e}, Manel Sanch\'{o}n, and Joel Spruck.
\newblock A priori estimates for semistable solutions of semilinear elliptic
  equations.
\newblock {\em Discrete Contin. Dyn. Syst.}, 36(2):601--609, 2016.

\bibitem{CS}
Daniele Castorina and Manel Sanch\'{o}n.
\newblock Regularity of stable solutions of {$p$}-{L}aplace equations through
  geometric {S}obolev type inequalities.
\newblock {\em J. Eur. Math. Soc. (JEMS)}, 17(11):2949--2975, 2015.

\bibitem{CR}
Michael~G. Crandall and Paul~H. Rabinowitz.
\newblock Some continuation and variational methods for positive solutions of
  nonlinear elliptic eigenvalue problems.
\newblock {\em Arch. Rational Mech. Anal.}, 58(3):207--218, 1975.

\bibitem{Dup}
Louis Dupaigne.
\newblock {\em Stable solutions of elliptic partial differential equations},
  volume 143 of {\em Chapman \& Hall/CRC Monographs and Surveys in Pure and
  Applied Mathematics}.
\newblock Chapman \& Hall/CRC, Boca Raton, FL, 2011.

\bibitem{Gelf}
I.~M. Gel\'fand.
\newblock Some problems in the theory of quasilinear equations.
\newblock {\em Amer. Math. Soc. Transl.}, 29:295--381, 1963.

\bibitem{i}
Kazuhiro Ishige and Paolo Salani.
\newblock Parabolic power concavity and parabolic boundary value problems.
\newblock {\em Math. Ann.}, 358(3-4):1091--1117, 2014.

\bibitem{Ishige}
Kazuhiro Ishige and Paolo Salani.
\newblock Parabolic {M}inkowski convolutions of solutions to parabolic boundary
  value problems.
\newblock {\em Adv. Math.}, 287:640--673, 2016.

\bibitem{JL}
D.~D. Joseph and T.~S. Lundgren.
\newblock Quasilinear {D}irichlet problems driven by positive sources.
\newblock {\em Arch. Rational Mech. Anal.}, 49:241--269, 1972/73.

\bibitem{Ned}
Gueorgui Nedev.
\newblock Regularity of the extremal solution of semilinear elliptic equations.
\newblock {\em C. R. Acad. Sci. Paris S\'{e}r. I Math.}, 330(11):997--1002,
  2000.

\bibitem{san}
Manel Sanch\'{o}n.
\newblock Boundedness of the extremal solution of some {$p$}-{L}aplacian
  problems.
\newblock {\em Nonlinear Anal.}, 67(1):281--294, 2007.

\bibitem{Vil}
Salvador Villegas.
\newblock Boundedness of extremal solutions in dimension 4.
\newblock {\em Adv. Math.}, 235:126--133, 2013.

\end{thebibliography}
\end{document}